\documentclass[twoside,a4paper,10pt]{amsart}

\usepackage{graphicx}

\usepackage{amsfonts,amssymb,amsmath}
\usepackage{dsfont}
\usepackage[all,cmtip]{xy} 

\allowdisplaybreaks[1]

\newcommand{\eL}{\mathbb{L}}
\newcommand{\Cm}{\mathcal{C}}

\renewcommand{\epsilon}{\varepsilon}

\newcommand{\ac}{\overline{\mathrm{ac}}}

\newcommand{\ord}{\mathrm{v}}

\newcommand{\Rr}{\mathbb{R}}
\newcommand{\Nn}{\mathbb{N}}
\newcommand{\Zz}{\mathbb{Z}}

\newcommand{\Qq}{\mathbb{Q}}

\newcommand{\Aa}{\mathbb{A}}

\newcommand{\Jac}{\mathrm{Jac}}

\newcommand{\set}[1]{\left\{#1\right\}}

\newcommand{\abs}[1]{\left| #1 \right|}

\newcommand{\kk}{\mathbf{k}}

\newcommand{\un}{\mathds{1}}

\newcommand{\Var}[1]{\mathrm{Var}_{#1}}

\newcommand{\K}[1]{\mathrm{\bf K}(#1)}

\newcommand{\calO}{\mathcal{O}}

\newcommand{\calD}{\mathcal{D}}

\newcommand{\FF}{\mathbb{F}}

\newcommand{\loc}{\mathrm{loc}}

\newcommand{\dbigcup}{\overset{.}{\bigcup}}

\theoremstyle{plain}
\newtheorem{nth}{theorem}[section]

\newtheorem{theorem}[nth]{Theorem}

\newtheorem{lemma}[nth]{Lemma}

\theoremstyle{definition}

\newtheorem{lemma-Definition}[nth]{Lemma-Definition}
\newtheorem{proposition-definition}[nth]{Proposition-Definition}

\theoremstyle{remark}

\begin{document}

\title{A motivic local Cauchy-Crofton formula}
\author{Arthur Forey}
\address{Sorbonne Universit\'es, UPMC Univ Paris 06, UMR 7586 CNRS, Institut Math\'ematique de Jussieu, F-75005 Paris, France}
\email{arthur.forey@imj-prg.fr}
\urladdr{https://webusers.imj-prg.fr/$\sim$arthur.forey/}
\thanks{This research was partially supported by ANR-15-CE40-0008 (D\'efig\'eo).}
\date{\today}


\begin{abstract}
In this note, we establish a version of the local Cauchy-Crofton formula for definable sets in Henselian discretely valued fields of characteristic zero. It allows to compute the motivic local density of a set from the densities of its projections integrated over the Grassmannian. 
\end{abstract}

\maketitle

\section{Introduction}

The aim of this note is to establish a motivic analogue of the local Cauchy-Crofton formula.  The classical Cauchy-Crofton formula is a geometric measure theory result stating that the volume of a set $X$ of dimension $d$ can be recovered by integrating over the Grassmannian the number of points of intersection of $X$ with affine spaces of codimension $d$, see for example \cite{federer_geometric_1969}. It has been used by Lion \cite{lion_densite_1998} to show the existence of the local density of semi-Pfaffian sets. Comte \cite{comte_formule_1999}, \cite{comte_equisingularite_2000} has established a local version of the formula for sets $X\subseteq \Rr^n$ definable in an $o$-minimal structure.  The formula states that the local density of such a set $X$ can be recovered by integrating over a Grassmannian the density of the projection of $X$ on subspaces. This allows him to show the continuity of the real local density along Verdier's strata in \cite{comte_equisingularite_2000}.

The local Cauchy-Crofton formula appears as a first step toward comparing the local Lipschitz-Killing curvature invariants and the polar invariants of a germ of a definable set $X\subseteq \Rr^n$. It is shown by Comte and Merle in \cite{comte_equisingularite_2008} that one can recover one set of invariants by linear combination of the other, see also \cite{comte_deformation_2015}.

A notion of local density for definable sets in Henselian valued field of characteristic zero has been develloped by the autor in \cite{MLD}. Our formula is a new step toward developing a theory of higher local curvature invariants in non-archimedean geometry.

A $p$-adic analogue has been developed by Cluckers, Comte and Loeser in \cite[Section 6]{cluckers_local_2012}. We will follow closely their approach. Our precise result appears as Theorem \ref{thm-localCC} at the end of Section \ref{Grass}.

\subsection*{Acknoledgements}
Many thanks to Georges Comte for encouraging me to work on this project and useful discussions. I also thanks Michel Raibaut for interesting comments.

\section{Motivic integration and local density}

We assume the reader is familiar with the notion of motivic local density developed by the autor in \cite{MLD} and in particular with Cluckers and Loeser's theory of motivic integration \cite{cluckers_constructible_2008}, \cite{cluckers_motivic_????}. See \cite[Section 2.7]{MLD} for a short summary of the theory.  

We adopt the notations and conventions of \cite{MLD}. In particular, we fix $\mathcal{T}$, a tame or mixed-tame theory of valued fields in the sense of \cite{cluckers_motivic_????}. Such a $\mathcal{T}$ always admits a Henselian discretely valued field of characteristic zero as a model.  Definable means definable without parameters in $\mathcal{T}$ and $K$ is (the underlining valued field of) a model of $\mathcal{T}$ with discrete value group and residue field $k$ enough saturated.

For example, we can take for $\mathcal{T}$ the theory of a discretely valued field of characteristic zero in the Denef-Pas language ; if the residue field is of characteristic $p>0$, one need to add the higher angular components. 

For each definable set $X$, Cluckers and Loeser define a ring of constructible motivic functions $\Cm(X)$, which include the characteristic functions of any definable set $Y\subseteq X$. 

For $\varphi\in\Cm(X)$ of support of dimension $d$ and that is integrable, they define
\[
\int^d \varphi\in \Cm(\set{*}).
\]
If $d=\dim(X)$, we drop the $d$ from the notation. If $\varphi$ is the characteristic function of some definable set $Y\subseteq X$, we denote the integral $\mu_d(Y)$.

If the residue field $k$ is algebraically closed, the target ring of motivic integration is equal to 
\[
\Cm(\set{*})={\K{\Var{k}}}_\loc:=\K{\Var{k}}\left[\eL^{-1},\left(\frac{1}{1-\eL^{-\alpha}}\right)_{\alpha\in \Nn^*}\right],
\]
where $\eL=[\Aa^1_k]$. 

Fix some definable set $X\subseteq K^m$ of dimension $d$, and $x\in K^m$, define
\[
\theta_n=\frac{\mu_d(X\cap B(x,n)}{\eL^{-nd}},
\]
where $B(x,n)$ is the ball of center $x$ and valuative radius $n$. 

There is some $e\in \Nn^*$ such that for each $i$, the subsequence $(\theta_{ke+i})_{k\in \Nn}$ converges to some $d_i\in \Cm(\set{*})$. Here $\Cm(\set{*})$ has a topology induced by the degree in $\eL$. We define the motivic local density of $X$ at $x$ to be
\[
\Theta_d(X,x)=\frac{1}{e}\sum_{i=0}^{e-1} d_i\in \Cm(\set{*})\otimes \Qq.
\]

It is shown in \cite{MLD} that one can compute the motivic local density on the tangent cone. Fix some $\Lambda \in \mathcal{D}$, where $\mathcal{D}=\set{\Lambda_{n,m}\mid n,m\in \Nn^*}$,
\[
\Lambda_{n,m}=\set{\lambda\in K^\times\mid \ac_m(\lambda)=1, \ord(\lambda)=0 \mod n}.
\]
The $\Lambda$-tangent cone of $X$ at $x$ is the definable set
\[
C_x^\Lambda(X)=\set{u\in K^n\mid \forall i\in \Zz, \exists y\in X, \exists \lambda\in \Lambda, \ord(y-x)\geq i, \ord(\lambda(y-x)-u)\geq i}.
\]
The $\Lambda$-tangent cone with multiplicities is a definable function $CM_x^\Lambda(X)\in\Cm(K^m)$, of support $C_x^\Lambda(X)$, well defined up to a set of dimension $<d$. For example, if $X\subseteq K^m$ is of dimension $m$, there is no muliplicity to take into account and $CM_x^\Lambda(X)$ is the characteristic function of $C_x^\Lambda(X)$. 

Theorem 3.25 and 5.12 from \cite{MLD} state that there is a $\Lambda$ such that for all $\Lambda'\subseteq \Lambda$, $C_x^\Lambda(X)=C_x^{\Lambda'}(X)$ and
\[
\Theta_d(X,x)=\Theta_d(CM_x^\Lambda(X),0).
\]

\section{Local constructible functions}

Consider a definable function $\pi : X\to Y$ between definable sets $X$ and $Y$ of dimension $n$. Recall form \cite{cluckers_motivic_????} the notation $\pi_!(\varphi)\in \Cm(Y)$ for any motivic constructible function $\varphi\in \Cm(X)$.

If $X$ is a definable subset of $K^{n}$ and $x\in K^n$, define the ring of germs of constructible motivic functions at $x$ by
$\Cm(X)_x:=\Cm(X)/\sim$, where $\varphi\sim\psi$ if there is an $r\in \Nn$ such that $\un_{B(x,r)}\varphi=\un_{B(x,r)}\psi$. in particular, if $\varphi\sim\psi$ are locally bounded, then $\Theta_d(\varphi,x)=\Theta_d(\psi,x)$ hence the local motivic density is defined on $\Cm(K^n)_x$.

Consider now a linear projection $\pi : K^n\to K^d$ and let $X\subseteq K^n$ a definable set of dimension $d$. Say that condition $(*)$ is satisfied if $\pi_{\vert X\cap B(x,r)}$ is finite-to-one for some $r\geq 0$. 

 Then define $\pi_{!,x}(\varphi)=\pi_!(\un_{B(x,r)}\varphi)\in \Cm(Y)_{\pi(x)}$ for $\varphi\in \Cm(X)_x$. It does not depend on the large enough $r$ chosen. Indeed, as $\pi : X\cap B(x,r) \to Y$ is finite-to-one, then there is $W_r\subseteq R_s^s$ finite with $X\cap B(x,r)\simeq W_r\times Y$ and the composite 
 \[W\times Y\overset{\pi}{\to} X\cap B(x,r)\to Y\]
 being the projection on $Y$. Hence (up to taking a finite definable partition of $X$), for some $r_0$, for any $r\geq r_0$, $W_r=W_{r_0}$ and $\pi(X\cap B(x,r)=\pi(Y)\cap B(\pi(x),\alpha r)$, which means that $\pi_!(\un_{B(x,r)}\varphi)\in \Cm(Y)/_{\pi(x)}$ is independent of $r$. 

\section{Grassmannians}
\label{Grass}
Fix a point $x\in K^n$ and view $K^n$ as a $K$-vector space with origin $0$. Then denote by $G(n,d)_K$ the Grassmannian of dimension $d$ subvector spaces of $K^n$. The canonical volume form on $G(n,d)_K$ invariant under $GL_n(\mathcal{O}_K)$-transformations induce a constuctible function $\omega_{n,d}$ on $G(n,d)$ invariant under $GL_n(\mathcal{O}_K)$ transformations, see \cite[Section 15]{cluckers_constructible_2008} for details. Since $G(n,d)_k$ is smooth and proper, the motivic volume of $G(n,d)_K$ is equal to the class $[G(n,d)_k]$ of $G(n,d)_k$ is the (localized) Grothendieck group of varieties over the residue field $k$. Denoting $\mathbb{F}_q$ the finite field with $q$ elements, it is known, see for example \cite{andrews_theory_1976}, that 
\[
\abs{G(n,d)(\mathbb{F}_q)}=\frac{(q^n-1)(q^n-q)\cdots(q^n-q^{r-1})}{(q^r-1)(q^r-q)\cdots(q^r-q^{r-1})}.
\]
The proof rely on the fact that 
\[
\abs{GL(n)(\FF_q)}=\abs{G(n,r)(\FF_q)}\abs{GL(r)(\FF_q)}q^{r(n-r)}\abs{GL(n-r)(\FF_q)}.
\]

Since the analog of this formula holds in $\K{\Var{k}}$, the same proof shows that
\[[G(n,d)_k]=\frac{(\eL^n-1)(\eL^n-\eL)\cdots(\eL^n-\eL^{r-1})}{(\eL^r-1)(\eL^r-\eL)\cdots(\eL^r-\eL^{r-1})}\in 
{\K{\Var{k}}}_\loc.\] 
Note that even if the right hand side can be written without denominator, hence as an element of $\K{\Var{k}}$, one has to work in ${\K{\Var{k}}}_\loc$ to show the equality with this method. In particular, this shows that $[G(n,d)_k]$ is invertible in ${\K{\Var{k}}}_\loc$. The motivic volume of $G(n,r)_K$ is then invertible. Hence we can normalize $\omega_{n,r}$ such that 
\[
1=\int_{V\in G(n,r)} \omega_{n,r}(V).
\]

For $V\in G(n,n-d)$, define $p_V : K^n\to K^n/V$ the canonical projection. We identify $K^n/V$ to $K^d$ as follows. There is some $g\in GL_n(\calO_K)$ such that $g(K^{n-d}\times\set{0}^d)=V$. We identify $K^n/V$ to $g(\set{0}^{n-d}\times K^{n-d})$. The particular choice of $g$ does not matter thanks to the change of variable formula. If $X$ is a dimension $d$ definable subset of $K^n$ then there is an dense definable subset $\Omega=\Omega(X,x)$ of $G(n,n-d)$ such that for every $V\in \Omega$, $\pi_V$ satisfies condition $(*)$. Indeed the tangent $K^\times$-cone of $X$ is of dimension at most $d$ and it suffices to set
\[
\Omega=\set{V\in G(n,n-d)\mid V\cap C_x^{K^\times}(X)=\set{0}},
\]
which is indeed dense in $G(n,n-d)$ In particular, for any $V\in \Omega$, $p_{!,x}(\varphi)$ is well defined for any $\varphi\in \Cm(X)_x$.

With these notations, we can now state our motivic local Cauchy-Crofton formula. Recall that $\Theta_d(X,x)\in \Cm(*)\otimes \Qq$ is the motivic local density of $X$ at $x$, 

\begin{theorem}[local Cauchy-Crofton]
\label{thm-localCC}
Let $X\subseteq K^n$ a definable set of dimension $d$ and $x\in K^n$. Then 
\[
\Theta_d(X,x)=\int_{V\in \Omega\subseteq G(n,n-d)} \Theta_d(p_{V!,x}(\un_X),0) \omega_{n,n-d}(V).
\]
\end{theorem}

By \cite[Proposition 3.8]{MLD}, we may assume $X=\overline{X}$. We can also assume $x=0$ and $0\in X$. Indeed, if $0\notin X$,  then both sides of the formula are 0.

\section{Tangential Crofton formula}

We start by proving the theorem in the particular case where $X$ is a $\Lambda$-cone.

\begin{lemma}
\label{lem-CC-singlecone}
Let $X$ be a definable $\Lambda$-cone with origin 0 contained in some $\Pi\in G(n,d)$. Then
\[
\Theta_d(X,0)=\int_{V\in \Omega\subseteq G(n,n-d)} \Theta_d(p_{V}(\un_X),0) \omega_{n,n-d}(V).
\]
\end{lemma}

\begin{proof}
Assume $\Lambda=\Lambda_{e,r}$. Fix some $V\in G(n,n-d)$ such that $\pi_V : \Pi \to K^d$ is bijective. As $X$ is a $\Lambda$-cone and $\pi_V$ is linear, $\pi_V(X)$ is also a $\Lambda$-cone. From the definition of local density, see also \cite[Remark 3.11]{MLD}, we have
\begin{equation}
\label{eqn-thetaX}
\Theta_d(X,0)=\frac{1}{e(1-\eL^{-d})}\sum_{i=0}^{e-1} \eL^{di} \mu_d(X\cap S(0,i))
\end{equation}
and 
\begin{equation}
\label{eqn-thetapiVX}
\Theta_d(\pi_V(X),0)=\frac{1}{e(1-\eL^{-d})}\sum_{j=0}^{e-1} \eL^{dj} \mu_d(X\cap S(0,j)).
\end{equation}
Let 
\[
A_i=\set{y\in \pi_V(X)\mid \exists x\in X\cap S(0,i),\exists \lambda\in \Lambda, y=\lambda\pi_V(x)}
\]
Then since $X$ is a $\Lambda$-cone and $\pi_V$ is bijective, we have a disjoint union
\[
\pi_V(X)\backslash\set{0}=\dbigcup_{i=0}^{e-1}A_i.
\]
Now set $B_i^j=A_i\cap S(0,j)$ and $D_i^V=\dbigcup_{j=0}^{e-1}t^{i-j}B_i^j$. 
The set $C_j$ is indeed a disjoint union since it is the image of $X\cap S(0,i)$ by the application
\[
\varphi_V : x\in X \mapsto t^{\ord(x)-\ord(\pi_V(x)} \pi_V(x).
\]
Indeed the function $\varphi_V$ restricted to $X\cap S(0,i)$ is a definable bijection of image $D_i^V$ since $\pi_V$ is linear and bijective on $\Pi$. 

By the change of variable formula, we have
\begin{equation}
\label{eqn-chgvar-varphi}
\mu_d(D_i^V)=\int_{X\cap S(0,i)} \eL^{-\ord(\Jac(\varphi_V(x)))}.
\end{equation}

By Fubini theorem we get
\[
\int_{V\in G(n,n-d)} \mu_d(D_i^V)\omega_{n,n-d}(V)=\int_{x\in X\cap S(0,i)}\int_{V\in G(n,n-d)} \eL^{-\ord(\Jac(\varphi_V(x)))}\omega_{n,n-d}(V).
\]
Set $C_i(x)=\int_{V\in G(n,n-d)} \eL^{-\ord(\Jac(\varphi_V(x)))}\omega_{n,n-d}(V)$. We claim that $C_i(x)$ is independent of $x\in S(0,i)$. Indeed, if $x,x'\in S(0,i)$, we can find some $g\in GL_n(\calO_K)$ such that $x'=gx$. Since $\omega_{n,n-d}$ is invariant under $GL_n(\calO_K)$-transformations and $\varphi_{V}(x)=\varphi_{gV}(x')$, by the change of variable formula we get that $C_i(x)$ is equal to
\[
\int_{V\in G(n,n-d)} \eL^{-\ord(\Jac(\varphi_V(x)))}\omega_{n,n-d}(V)=\int_{V'\in G(n,n-d)} \eL^{-\ord(\Jac(\varphi_{V'}(gx)))}\omega_{n,n-d}(V').
\]
Moreover, it is independent of $i$ by linearity of $\pi_V$, hence we denote it by $C$ and we have
\begin{equation}
\label{eqn-finalmuDiV}
\int_{V\in G(n,n-d)} \mu_d(D_i^V)\omega_{n,n-d}(V)=C\mu_d(X\cap S(0,i).
\end{equation}

We also have 
\begin{equation}
\label{eqn-changevarBij}
\eL^{dj}\mu_d(B_i^j)=\eL^{di}\mu_d(t^{i-j}B_i^j),
\end{equation}
 hence
\begin{align}
\sum_{j=0}^{e-1} \eL^{dj} \mu_d(\pi_V(X)\cap S(0,j))
&=\sum_{j=0}^{e-1}\sum_{i=0}^{e-1}\eL^{dj} \mu_d(B_i^j) \notag\\
&=\sum_{i=0}^{e-1}\sum_{j=0}^{e-1} \eL^{di} \mu_d(t^{i-j}B_i^j)\notag\\
&=\sum_{i=0}^{e-1}\eL^{di} \mu_d(D_i^V).\label{eqn-sumThetaDiV}
\end{align}

Combining Equations \ref{eqn-thetapiVX}, \ref{eqn-sumThetaDiV} and \ref{eqn-finalmuDiV}, we get
\[
\int_{V\in G(n,n-d)}\Theta_d(\pi_V(X),0)\omega_{n,n-d}(V)=C\frac{1}{e(1-\eL^{-d})}\sum_{i=0}^{e-1} \eL^{di} \mu_d(X\cap S(0,i)).
\]
This last expression is equal to $=C\Theta_d(X,0)$. We find $C=1$ by computing both sides of the previous equality with $X=\Pi$. 
\end{proof}

This following lemma is the motivic analogous of the classical spherical Crofton formula, see for example \cite[Theorem  3.2.48]{federer_geometric_1969}. See also \cite[Remark 6.2.4]{cluckers_local_2012} for a reformulation in the $p$-adic case. 

\begin{lemma}
\label{lem-CC-generalLamcone}
Let $X$ be a definable $\Lambda$-cone with origin 0.  Then
\[
\Theta_d(X,0)=\int_{V\in G(n,n-d)} \Theta_d(p_{V!,0}(\un_X),0) \omega(V).
\]
\end{lemma}

\begin{proof}
We only have to modify slightly the proof of Lemma \ref{lem-CC-singlecone}. Indeed, we use now the function
\[
\varphi_V : x\in X \mapsto t^{\ord(x)-\ord(\pi_V(x)} \pi_V(x).
\]
restricted to the smooth part of $X$. It is now longer injective on $X\cap S(0,i)$, however the motivic volume of the fibers is taken into account in $p_{V!,0}(X)$. Hence we get similarly
\[
\int_{V\in G(n,n-d)}\Theta_d(\pi_{V!,0}(X),0)\omega_{n,n-d}(V)=C\frac{1}{e(1-\eL^{-d})}\sum_{i=0}^{e-1} \eL^{di} \mu_d(X\cap S(0,i)),
\]
which is equal to $C\Theta_d(X,0)$. Once again, we find $C=1$ by computing both sides for $X$ a vector space of dimension $d$. 
\end{proof}

\section{General case}

Before proving Theorem \ref{thm-localCC}, we need a technical lemma. 

\begin{lemma}
\label{lem-partitionCLampVinj}
Let $X\subseteq K^n$ be a definable set of dimension $d$ and $V\in G(n,n-d)$ such that the projection $p_V : C_0^\Lambda(X) \to K^d$ is finite-to-one. Then there is a definable $\kk$-partition of $X$ such that for each $\kk$-part $X_\xi$, there is a $\xi$-definable set $C_\xi$ of dimension less than $d$ such that $p_V$ is injective on $C_0^\Lambda(X_\xi)\backslash C_\xi$. 
\end{lemma}

\begin{proof}
We can assume $\Lambda=K^\times$. 
As the projection $p_V :  C_0^\Lambda(X) \to K^d$ is finite-to-one, by finite $b$-minimality one can find a $\kk$-partition of $C_0^\Lambda(X)$ such that $p_V$ is injective on each $\kk$-part of $C_0^\Lambda(X)$. For a $\kk$-part $C_0^\Lambda(X)_\xi$, define $B_\xi$ to be the $\xi$-definable subset of $K^n$ defined as the union of lines $\ell$ passing through 0 such that the distance between $\ell\cap S(0,0)$ and $C_0^\Lambda(X)_\xi\cap S(0,0)$ is strictly smaller than the distance between $\ell\cap S(0,0)$ and $C_0^\Lambda(X)_{\xi'}\cap S(0,0)$ for every $\xi'\neq \xi$. Set $X_\xi=X\cap B_\xi$. Then setting $Y=X\backslash \cup_{\xi}X_\xi$, we have $C_0^\Lambda(Y)$ empty and $C_0^\Lambda(X_\xi)\subseteq \overline{C_0^\Lambda(X)_\xi}$. Hence we set $C_\xi=\overline{C_0^\Lambda(X)_\xi}\backslash C_0^\Lambda(X)_\xi$ and we have $X_\xi$ and $C_\xi$ as required. 
\end{proof}

\begin{proof}[Proof of Theorem \ref{thm-localCC}]
From \cite[Theorem 3.25]{MLD}, there is a $\Lambda\in \calD$ such that $\Theta_d(X,0)=\Theta_d(CM_0^\Lambda(X),0)$. As in the proof of \cite[Theorem 3.25]{MLD}, by \cite[Proposition 2.14 and Lemma 3.9]{MLD}, we can assume $X$ is the graph of a 1-Lipschitz function defined on some definable set $U\subset K^d$. In this case, $CM_0^\Lambda(X)=\un_{C_0^\Lambda(X)}$. From
Lemma \ref{lem-CC-generalLamcone}, we have
\[
\Theta_d(C_0^\Lambda(X),0)=\int_{V\in G(n,n-d)} \Theta_d(p_{V!,0}(C_0^\Lambda(X)),0) \omega_{n,n-d}(V).
\]
Hence we need to show that for every $V$ in a dense subset of $G(n,n-d)$, 
\[
\Theta_d(p_{V!,0}(C_0^\Lambda(X)),0)=\Theta_d(p_{V!,0}(X),0).
\]
We can find a $\kk$-partition of $X$ such that $p_V$ is injective on the $\kk$-parts. Replace $X$ by one of the $\kk$-parts and suppose then that $p_V$ is injective on $X$. 

Fix a $V\in G(n,n-d)$ such that $p_V$ is injective on $C_0^\Lambda(X)$. By Lemma \ref{lem-partitionCLampVinj}, there is a $\kk$-partition of $X$ (depending on $V$) such that for each $\kk$-part $X_\xi$ there is a $\xi$-definable set $C_\xi$ of dimension less that $d$ such that $p_V$ is injective on $C_0^\Lambda(X_\xi)\backslash C_\xi$. By a new use of \cite[Lemma 3.9]{MLD}, it suffices to show that 
\[
\Theta_d(p_{V!,0}(C_0^\Lambda(X_\xi)),0)=\Theta_d(p_{V!,0}(X_\xi),0).
\]
As $p_V$ is injective on $X$ and $C_0^\Lambda(X_\xi)\backslash C_\xi$, we have $p_{V!,0}(X)=\un_{p_V(X)}$ and
\[p_{V!,0}(C_0^\Lambda(X_\xi))=\un_{p_V(C_0^\Lambda(X_\xi)\backslash C_\xi)}+p_{V!,0}(C_\xi).
\]
We have $\Theta_d(p_{V!,0}(C_\xi),0)=0$ for dimensional reasons. As $C_0^\Lambda(p_V(X))=p_V(C_0^\Lambda(X))$ the result follows from \cite[Proposition 5.2]{MLD}.
\end{proof}

\bibliographystyle{abbrv}

\end{document}